\documentclass[a4paper,12pt]{amsart}
\usepackage{fullpage}
\usepackage{enumerate}
\usepackage{amssymb}
\usepackage{amsmath}
\usepackage{latexsym}
\usepackage{amsthm}
\usepackage{xypic}
\usepackage{xcolor}
\usepackage[english]{babel}
\usepackage[backref]{hyperref}
\usepackage{mathtools, nccmath}

\newtheorem{theorem}{Theorem}[section]

\newtheorem{proposition}[theorem]{Proposition}

\theoremstyle{definition}

\theoremstyle{remark}
\newtheorem{remark}[theorem]{Remark}
\newtheorem*{claim}{Claim}

\theoremstyle{question}

\newtheorem*{theorem*}{Theorem}
\numberwithin{equation}{section}

\newcommand{\Q}{\mathbb{Q}}

\begin{document}

\title{Constructing totally $p$-adic numbers of small height}

\author{S. Checcoli}
\address{S.~Checcoli, Institut Fourier, Universit\'e Grenoble Alpes, 100 rue des Math\'ematiques,  38610 Gi\`eres, France}
\email{sara.checcoli@univ-grenoble-alpes.fr}

\author{A. Fehm}
\address{A.~Fehm, Institut f\"ur Algebra, Fakult\"at Mathematik, Technische Universit\"at Dresden, 01062 Dresden, Germany}
\email{arno.fehm@tu-dresden.de}

\begin{abstract} 
Bombieri and Zannier gave an effective construction of algebraic numbers of small height inside the maximal Galois extension of the rationals which is totally split at a given finite set of prime numbers. They proved, in particular, an explicit upper bound for the lim inf of the height of elements in such fields.
We generalize their result in an effective way to maximal Galois extensions of number fields with given local behaviour at finitely many places.
\end{abstract}
\maketitle
\section{Introduction}
\noindent
Let $h$ denote the absolute logarithmic Weil height on 
the field $\overline{\Q}$ of algebraic numbers. 
We are interested in explicit height bounds for elements of $\overline{\Q}$ with special local behaviour at a finite set of primes.
The first result in this context is due to Schinzel \cite{Sch} who proved a height lower bound for elements in the field of totally real algebraic numbers  $\Q^{\rm tr}$, the maximal Galois extension of $\Q$ in which the infinite prime splits totally. More precisely, he showed that every $\alpha\in \Q^{\rm tr}$ has either $h(\alpha)=0$ or \[h(\alpha)\geq \frac{1}{2}\log\left(\frac{1+\sqrt{5}}{2}\right).\]
Explicit upper and lower bounds for the limit infimum of the height of algebraic integers in $\Q^{\rm tr}$ are given in \cite{Smy80,Smy81, Fla96}.

In \cite{BZ} Bombieri and Zannier
investigate the analogous problem for the $p$-adic numbers.
More precisely, in \cite[Theorem 2]{BZ} they prove the following:
\begin{theorem}[Bombieri--Zannier]\label{BZ_lower}
Let $p_1,\dots,p_n$ be distinct prime numbers,
for each $i$ let $E_i$ be a finite Galois extension of $\mathbb{Q}_{p_i}$,
and $L$ the maximal Galois extension of $\mathbb{Q}$ contained in all $E_i$.
Denote by $e_i$ and $f_i$ the ramification index and inertia degree of $E_i/\mathbb{Q}_{p_i}$.
Then
$$
 \liminf_{\alpha\in L} h(\alpha) \;\geq\; \frac{1}{2}\cdot \sum_{i=1}^n\frac{\log(p_i)}{e_i(p_i^{f_i}+1)}.
$$
\end{theorem}
In the special case $E_i=\mathbb{Q}_{p_i}$, Bombieri and Zannier in \cite[Example 2]{BZ} show that the lower bound in Theorem \ref{BZ_lower} is almost optimal. More precisely: 
\begin{theorem}[Bombieri--Zannier]\label{BZ_upper}
Let $p_1,\dots,p_n$ be prime numbers and let $L$ be the maximal Galois extension
of $\mathbb{Q}$ contained in all $\mathbb{Q}_{p_i}$. Then
$$
 \liminf_{\alpha\in L} h(\alpha) \;\leq\; \sum_{i=1}^n\frac{\log(p_i)}{p_i-1}.
$$
\end{theorem}
Other proofs, refinements and generalizations were given in
\cite{Fil, Pot,FiliPetsche, FP, PS19}
See also \cite{Smy07} for a general survey on the height of algebraic numbers. 
In Remark \ref{rem-Fili} we will discuss in detail the contribution \cite{Fil} and how it compares to our work.

The goal of this note is to generalize in an effective way the upper bound Theorem \ref{BZ_upper}
to general $E_i$,
and to further replace the base field $\mathbb{Q}$ by an arbitrary number field. Our main result is the following:

\begin{theorem}\label{mainthm}
Let $K$ be a number field and let $\mathfrak{p}_1,\dots,\mathfrak{p}_n$ be distinct primes ideals of the ring of integers $\mathcal{O}_K$ of $K$.  
For each $i$, 
let $E_i$ be a finite Galois extension of the completion $F_i$ of $K$ at $\mathfrak{p}_i$.
Denote by $e_i$ and $f_i$ the ramification index and the relative inertia degree of $E_i/F_i$
and write $q_i=|\mathcal{O}_K/\mathfrak{p}_i|=p_i^{{f(\mathfrak{p}_i|p_i)}}$.
Then for the maximal Galois extension $L$ of $K$ contained in all $E_i$,
\[
 \liminf_{\alpha\in L} h(\alpha) \;\leq\; \sum_{i=1}^n \frac{f(\mathfrak{p}_i|p_i)}{[K:\mathbb{Q}]}\cdot\frac{\log(p_i)}{e_i(q_i^{f_i}-1)}.
\]
More precisely, let
\[
 C=\max\left\{[K:\Q], |\Delta_K|, \max_i (e_i f_i), \max_i q_i^{f_i}\right\}
\] 
where $\Delta_K$ is the absolute discriminant of $K$.
Then for every $0<\epsilon<1$ there exist infinitely many $\alpha\in\mathcal{O}_L$ of height
\begin{equation}\label{eqn:thm1}
 h(\alpha)\leq \sum_{i=1}^n \frac{f(\mathfrak{p}_i|p_i)}{[K:\mathbb{Q}]}\cdot\frac{\log(p_i)}{e_i(q_i^{f_i}-1)}+13 nC^{2n+2}\frac{\log \left([K(\alpha):K]\right)}{[K(\alpha):K]} +\begin{cases}0,&n=1\\n\epsilon,&n>1\end{cases}.
\end{equation}
Namely, for every $\rho\geq 3C^n$ there exists such $\alpha$ of degree
\begin{equation}\label{eqn:thm2}
 \rho\leq [K(\alpha):K] \leq \begin{cases}C\rho,&n=1\\\rho^{\frac{(4 \log C)^{n+1}}{\log^n(1+\epsilon)}},&n>1\end{cases}.
\end{equation}
\end{theorem}

Note that in the special case $K=\mathbb{Q}$ and $E_i=\mathbb{Q}_{p_i}$
we reobtain Theorem \ref{BZ_upper}, 
except that Theorem \ref{mainthm} appears stronger in that the
result is effective and the $\liminf$ can be taken over algebraic integers.
However, 
an inspection of the proof of Bombieri and Zannier 
shows that it is effective as well and does in fact produce algebraic integers. 

\begin{remark}\label{rem-Fili}
Theorem \ref{mainthm} provides an effective version of a result of Fili \cite[Theorem 1.2]{Fil}.
The bound in that result seems to differ from ours by the factor  $e(\mathfrak{p}_i|p_i)$,
and \cite[Theorem 1.1]{Fil} (and similarly \cite[Theorem 9]{FiliPetsche}) also states a variant of Theorem~\ref{BZ_lower}
which contradicts our Theorem \ref{mainthm},
but
according to Paul Fili (personal communication)
this is merely an error in normalization in \cite{Fil} and \cite{FiliPetsche}
that became apparent when comparing to our result, 
and the $e_v$ in the denominator of Theorems 1.1, 1.2, and Conjecture 1 of \cite{Fil} (and similarly in the statements of \cite{FiliPetsche})
should have been the absolute instead of the relative ramification index. 
When this correction is made,
the lower bound of \cite[Theorem 1.1]{Fil} agrees with the one in \cite[Theorem 13]{FP},
and the upper bound of \cite[Theorem 1.2]{Fil} agrees with the one in Theorem \ref{mainthm}.

In any case, Fili's proof of \cite[Theorem 1.2]{Fil} uses capacity theory on analytic Berkovich spaces and does not provide explicit bounds on the degree and the height of a sequence of integral elements in the $\liminf$. Instead, our effective proof is more elementary and is inspired by Bombieri and Zannier's effective proof of Theorem \ref{BZ_upper}. 
To the best of our knowledge,
Theorem \ref{mainthm} is the only result currently available that gives a bound 
on the height in terms of the degree of such a sequence of $\alpha$,
except for the case where $K=\mathbb{Q}$ and $E_i=\mathbb{Q}_{p_i}$ for all $i$,
where such a bound can be deduced from \cite{BZ}.

We also remark that our use of \cite[Theorem 1.2]{Fil} in \cite{CF} is limited to the cases where \cite[Theorem 1.2]{Fil} agrees with  Theorem \ref{mainthm}.
\end{remark}

The paper is organised as  follows. 
In Section \ref{prel} we collect all the preliminary results needed to prove Theorem \ref{mainthm}, namely: 
a consequence of Dirichlet's theorem on simultaneous approximation (Proposition \ref{Dir}), 
a bound for the size of representatives in quotient rings of rings of integers (Proposition \ref{lem:small_rep}), 
a variant of Hensel's lemma (Proposition \ref{val-BZ}),
a bound for the height of a root of a polynomial defined over a number field in terms of its coefficients (Proposition \ref{H-min-root}),
and a construction of special Galois invariant sets of representatives of residue rings of local fields (Proposition \ref{A_i}).

The proof of Theorem \ref{mainthm} is carried out in Section \ref{sec:main}. We briefly sketch it here for clarity.
Following Bombieri and Zannier's strategy, given  $\rho\geq 3C^n$  we construct a monic irreducible polynomial $g\in\mathcal{O}_K[X]$ such that
\begin{enumerate}[(i)]
\item\label{d-g} its degree is upper and lower bounded in terms of $\rho$ as the degree of $\alpha$ in Theorem~\ref{mainthm},
\item\label{c-g} the complex absolute value of all conjugates of its coefficients is sufficiently small, and
\item\label{r-g} all its roots are contained in all $E_i$.
\end{enumerate}
In Bombieri and Zannier's proof of Theorem \ref{BZ_upper}, \eqref{d-g} and \eqref{c-g} were achieved by using the Chinese Remainder Theorem to deform the polynomial $\prod_{i=1}^{\rho} (X-i)$ into an irreducible polynomial of the same degree with coefficients small enough to give the desired bound for the height of the roots. Then a variant of Hensel's lemma was applied to show that the roots of the constructed polynomial are still in $\Q_{p_i}$ for each $i$.

In our generalisation, 
the degree of the polynomial is carefully chosen 
to obtain \eqref{d-g}
in Section \ref{sec:degree}
via Proposition \ref{Dir} (necessary only if $n>1$, which leads to the better bounds in the case $n=1$).
The polynomial $g$ satisfying \eqref{c-g} is then constructed in Section \ref{constr-G}:
We start with polynomials $\prod_{\alpha\in \tilde A_i} (X-\alpha)$, where now $\tilde A_i\subseteq \mathcal{O}_{E_i}$ is a set constructed using Proposition \ref{A_i}.
These polynomials are then merged into an irreducible polynomial $g$ by applying the Chinese Remainder Theorem and Proposition \ref{lem:small_rep} to bound the size of its coefficients.
Property \eqref{r-g}  is verified in Section \ref{roots-g}, using  Proposition \ref{val-BZ}. Finally, Proposition \ref{H-min-root} is applied to show that $g$ has a root $\alpha$ of height bounded from above as desired.

\section{Notation and preliminaries}\label{prel}
\noindent
We fix some notation. If $K$ is a number field or a non-archimedean local field we let $\mathcal{O}_K$ denote the ring of integers of $K$.
For an ideal $\mathfrak{a}$ of $\mathcal{O}_K$ 
we denote by $N(\mathfrak{a})=|\mathcal{O}_K/\mathfrak{a}|$ its norm.
For a nonzero prime ideal $\mathfrak{p}$ of $\mathcal{O}_K$,
we denote by $v_\mathfrak{p}$ the discrete valuation on $K$
with valuation ring $(\mathcal{O}_K)_{\mathfrak{p}}$
normalized such that $v_\mathfrak{p}(K^\times)=\mathbb{Z}$.
If $L/K$ is an extension of number fields and $\mathfrak{P}$ is a prime ideal of $\mathcal{O}_L$
lying above a prime ideal $\mathfrak{p}$ of $\mathcal{O}_K$
we denote by $e(\mathfrak{P}|\mathfrak{p})$ and $f(\mathfrak{P}|\mathfrak{p})$
the ramification index and the inertia degree.
For an extension $E/F$ of non-archimedean local fields we denote 
the ramification index and the inertia degree also by $e(E/F)$ and $f(E/F)$.

\subsection{Auxiliary results}
In this section we collect the preliminary results we need to prove Theorem \ref{mainthm}. These results are not related to each other and we list them in this section following their order of appearance in the proof of Theorem \ref{mainthm}.

\begin{proposition}\label{Dir}
Let $x_1,\dots,x_n$ be integers greater than $1$.
For every $\rho\geq 3$ and $0<\epsilon<1$ there exist positive integers $r,k_1,\dots,k_n$ such that $r\geq \rho$ and, for all $i$, $r\leq x_i^{k_i}\leq (1+\epsilon)r$  and
\[k_i\leq \frac{2^{2n+1}\log^n(\max_jx_j) \log(\rho)}{ \log(x_i) \log^n(1+\epsilon)}.\]
\end{proposition}

\begin{proof}
Say $x_1=\max_ix_i$.
Let $\alpha_i=2\log(\rho)/\log(x_i)$ and $Q=\lceil 2\log(x_1)/\log(1+\epsilon)\rceil$.
By the simultaneous Dirichlet approximation theorem \cite[Chapter II, Section 1, Theorem 1A]{Schm} 
there exist  positive integers $q,k_1,\dots,k_n$ with $1\leq q< Q^n$  such that $|q\alpha_i-k_i|\leq Q^{-1}$ for all $i$, and thus $|2\log(\rho) q-\log(x_i^{k_i})|\leq \log(1+\epsilon)/2$.
Letting $r=\min_i x_i^{k_i}$, one has $\log(r)\geq 2\log(\rho) q-1\geq \log(\rho)$ for $q\geq 1$ and $\rho\geq 3$. In addition,  for all $i$, $0\leq\log(x_i^{k_i})-\log(r)\leq\log(1+\epsilon)$,
hence $r\leq x_i^{k_i}\leq (1+\epsilon)r$. Finally $k_i\leq q \alpha_i+1 \leq 2 q \alpha_i\leq   2\log(\rho)Q^n\log(x_i)^{-1}$ and replacing $Q$ we get the desired bound.
\end{proof}

The next proposition deals with bounds for the absolute value of small representatives for quotient rings.

\begin{proposition}\label{lem:small_rep}
{Let $K$ be a number field of degree $m=[K:\Q]$.
Given a nonzero ideal $\mathfrak{a}$ of $\mathcal{O}_K$, there exists a set of representatives $A$ of $\mathcal{O}_K/\mathfrak{a}$ such that, for every $a\in A$ and every $\sigma\in{\rm Hom}(K,\mathbb{C})$, one has
\[|\sigma(a)|\leq \delta_K N(\mathfrak{a})^{1/m}\]
where $\delta_K=m^{\frac{3}{2}}{2}^{\frac{m(m-1)}{2}}\sqrt{|\Delta_K|}$.}
\end{proposition}

\begin{proof}
This is an immediate consequence of well-known results on lattice reduction. For instance, \cite[Proposition 15]{BFH} gives
for every $\alpha\in\mathcal{O}_K$, an element $a\in\mathcal{O}_K$ with $\alpha-a\in\mathfrak{a}$
such that \[\sqrt{\sum_{\sigma}|\sigma(a)|^2}\leq m^{\frac{3}{2}}{\ell}^{\frac{m(m-1)}{2}}\sqrt{|\Delta_K|} N(\mathfrak{a})^{1/m}\]
 where  $\ell$ depends on certain parameters {$\eta\in\ (1/2,1), \delta\in (\eta^2,1)$ and $\theta>0$ coming from applying a variant of the LLL-reduction algorithm as in \cite[Theorem 5.4]{Cha} and \cite[Theorem 7]{NSV} (see also \cite[\S 4, p.595]{BFH}) to a $\mathbb{Z}$-basis of $\mathfrak{a}$. In particular, choosing $\eta=2/3,\delta=7/9$ and $\theta=(\sqrt{19}-4)/3$, we have $\ell=2$, which gives the claimed upper bound  for $|\sigma(a)|$.}
\end{proof}

The following proposition is a variant of Hensel's lemma.
\begin{proposition}\label{val-BZ}
Let $E$ be a finite extension of $\mathbb{Q}_p$, $\mathfrak{P}$ the maximal ideal of $\mathcal{O}_E$ and $v=v_\mathfrak{P}$.
Let $f\in E[X]$ and $x_0\in E$.
Assume there exist $a,b\in\mathbb{Z}$ such that 
\begin{enumerate}[(i)]
\item\label{ci} $v(f(x_0))>a+b$, 
\item\label{cii} $v(f'(x_0))\leq a$, 
\item\label{ciii}
$v(f^{(\nu)}(x_0)/\nu!)\geq a-(\nu-1)b$ for every $\nu\geq 2$.\end{enumerate}
Then there exists $x\in E$ with $f(x)=0$ and $v(x-x_0)>b$. 
\end{proposition}
\begin{proof}
This can be proved precisely as the special case $E=\mathbb{Q}_p$ in \cite[Lemma 1]{BZ}. 
Alternatively, one can reduce this to one of the standard forms of Hensel's lemma as follows.
Let $\beta\in E$ with $v(\beta)=b$. 
Then $g(X):=(\beta f'(x_0))^{-1}f(\beta X+x_0)$
is in $\mathcal{O}_E[X]$ by $(i)$-$(iii)$
and has a simple zero $X=0$ modulo $\mathfrak{P}$ by $(i)$ and $(ii)$,
hence by Hensel's lemma $g$ has a zero $x'\in\mathfrak{P}$,
and $x=\beta x'+x_0$ is then the desired zero of $f$.
\end{proof}

The final proposition in this subsection gives a bound for the height of the roots of a polynomial with small algebraic coefficients.
\begin{proposition}\label{H-min-root}
Let $K$ be a number field and
let $f(X)=X^m+a_{m-1}X^{m-1}+\ldots +a_0\in \mathcal{O}_K[X]$.
If $B\geq 1$ with $|\sigma(a_i)|<B$ for every $i$ and every $\sigma\in{\rm Hom}(K,\mathbb{C})$,
then $f$ has a root $\alpha$ with 
\[
 h(\alpha)\leq \frac{\log(B\sqrt{m+1})}{m}.
\]
\end{proposition}

\begin{proof}
Let $M_K=M_K^0\cup M_K^{\infty}$ be the set of (finite and infinite) places of $K$ and let $d=[K:\Q]$. 
For a place $v\in M_K$, denote by $d_v=[K_v:\Q_v]$ the local degree.
Let 
\[
 \hat{h}(f)=\log\left(\prod_{v\in M_K}M_v(f)^{d_v/d}\right)
\]
where if $v$ is non-archimedean $M_v(f)=\max_i(|a_i|_v)$, while if $v$ is archimedean and corresponds to the embedding $\sigma\in{\rm Hom}(K,\mathbb{C})$, $M_v(f)$ is the Mahler measure $M(\sigma(f))$ of the polynomial $\sigma(f)$. 

By \cite[Appendix A, section A.2, pag. 210]{Zan} we have that, if $\alpha_1,\ldots, \alpha_m\in \overline{\Q}$ are the roots of $f$ (with multiplicities), then  $\hat{h}(f)=\sum_{i=1}^m h(\alpha_i)$
where $h$ denotes the usual logarithmic Weil height.
Thus, if $\alpha$ is a root of $f$ of minimal height, 
then 
\begin{equation}\label{h-root}
h(\alpha)\leq \frac{\hat{h}(f)}{m}.
\end{equation}
Let $\sigma_1,\ldots,\sigma_r$ and $\tau_1,\overline{\tau_1},\ldots, \tau_s,\overline{\tau_s}$ be, respectively, the real and pairwise conjugate complex embeddings of $K$ in $\mathbb{C}$, so that $d=r+2s$. 
As $f$ has coefficients in $\mathcal{O}_K$, 
$M_v(f)\leq1$ if $v$ is non-archimedean and we have that
\begin{align*}
 \hat{h}(f)&\leq\log\left(\prod_{v\in M_K^{\infty}}M_v(f)^{d_v/d}\right)=\log\left(\prod_{i=1}^r M(\sigma_i(f))^{1/d}\cdot \prod_{j=1}^s M(\tau_j(f))^{2/d}\right)=\\
&=\log\left(\prod_{\sigma\in{\rm Hom}(K,\mathbb{C})}M(\sigma(f))^{1/d}\right).
\end{align*}
By \cite[Section 3.2.2, formula (3.7)]{Zan} and by our hypothesis on $B$, we have that $M(\sigma(f))\leq B\sqrt{m+1}$ for all $\sigma\in{\rm Hom}(K,\mathbb{C})$, thus $\hat{h}(f)\leq \log(B\sqrt{m+1})$ and, plugging this bound into \eqref{h-root}, we conclude.
\end{proof}

\subsection{Representatives of residue rings of local fields}\label{sec-aux-p}
This subsection contains the technical key result needed to construct the local polynomials in the proof of Theorem \ref{mainthm}.
Let $E/F$ be a Galois extension of non-archimedean local fields with Galois group $G$.
{Let $\mathfrak{p}$ be the maximal ideal of $\mathcal{O}_F$,}
$\mathfrak{P}$ be the maximal ideal of $\mathcal{O}_E$
and for $k\in\mathbb{N}$ denote by
$\pi_k:\mathcal{O}_E\rightarrow\mathcal{O}_E/\mathfrak{P}^k$  the residue map.

It is known that one can always find a $G$-invariant set of representatives of the residue field $\mathcal{O}_E/\mathfrak{P}$,
e.g.~the Teichm\"uller representatives.
As long as the ramification of $E/F$ is tame,
one can also find $G$-invariant sets of representatives of each
residue ring $\mathcal{O}_E/\mathfrak{P}^k$,
but if the ramification is wild, 
this is not necessarily so. 
We will therefore work with the following substitute
for such a $G$-invariant set of representatives:

\begin{proposition}\label{A_i}
Let $E/F$ be a Galois extension of non-archimedean local fields
and define $G,\mathfrak{p},\mathfrak{P},\pi_k$ as above.
Let $d$ be a multiple of $|G|$.
There exists a constant $c$ such that for every $k$ there is  $A\subseteq\mathcal{O}_E$
such that
\begin{enumerate} 
\item $A$ is $G$-invariant,
\item all orbits of $A$ have length $|G|$,
\item $\pi_k|_{A}:A\rightarrow\mathcal{O}_E/\mathfrak{P}^k$ is $d$-to-1 and onto, and
\item $\pi_{k+c}|_{A}$ is injective.
\end{enumerate}
Moreover, if $F$ is a $p$-adic field, one can choose
\[
  c\leq e(\mathfrak{P}|\mathfrak{p})\left(d+|G|+\frac{e(\mathfrak{p}|p)}{p-1}+1\right).
\]
\end{proposition}

\begin{proof}
Note that $G$ naturally acts on $\mathcal{O}_E$ and on $\mathcal{O}_E/\mathfrak{P}^k$, and that $\pi_k$ is $G$-equivariant.
Fix some primitive element $\alpha\in\mathcal{O}_E^\times$ of $E/F$ 
and a uniformizer $\theta\in\mathcal{O}_F$ of $v_\mathfrak{p}$,
and
let 
\begin{eqnarray*}
 e&=&e(\mathfrak{P}|\mathfrak{p}),\\
 c_0&=&\max_{1\neq\sigma\in G} v_\mathfrak{P}(\alpha-\sigma\alpha), (\mbox{with } c_0=0\mbox{ if }G=1), \\
 c_1&=&\lceil|G|+c_0/e\rceil, \mbox{ and }\\
 c &=& e(d+c_1).
\end{eqnarray*} 
Let $k\in\mathbb{N}$ be given.
The desired set $A$ is obtained by applying the following Claim
in the case $X=\mathcal{O}_E/\mathfrak{P}^k$:
\begin{claim}
For every $G$-invariant subset
$X\subseteq\mathcal{O}_E/\mathfrak{P}^k$
there exists a $G$-invariant subset $A\subseteq\mathcal{O}_E$ 
with all orbits of length $|G|$ such that
$\pi_{k+c}|_{A}$ is injective
and $\pi_k|_{A}$ is $d$-to-$1$ onto $X$.
\end{claim}
We prove the Claim by induction on $|X|$:
If $X=\emptyset$, $A=\emptyset$ satisfies the claim.
If $X\neq\emptyset$ take $x\in X$ and let $X'=X\setminus Gx$,
where $Gx$ denotes the orbit of $x$ under $G$.
By the induction hypothesis there exists 
$A'\subseteq\mathcal{O}_E$ satisfying the claim for $X'$.
Choose $a\in\pi_k^{-1}(x)$
and let $k_0=\lceil\frac{k}{e}\rceil$.
Then
$$
 n_0:=\min\{n\geq0: v_\mathfrak{P}(a-\sigma a)\neq e(k_0+n)+v_\mathfrak{P}(\alpha-\sigma\alpha)\;\forall 1\neq\sigma\in G\} < |G|,
$$
as $v_\mathfrak{P}(a-\sigma a)-v_\mathfrak{P}(\alpha-\sigma\alpha)$ attains less than $|G|$ many distinct values.
Thus for $1\neq\sigma\in G$,
\begin{eqnarray*}
 v_\mathfrak{P}( (a+\theta^{n_0+k_0}\alpha)-\sigma(a+\theta^{n_0+k_0}\alpha) )
 &=&  \min\{v_\mathfrak{P}(a-\sigma a),e(n_0+k_0)+v_\mathfrak{P}(\alpha-\sigma\alpha)\}\\
  &\leq&e(n_0+k_0)+c_0\\&<& k+ec_1,
\end{eqnarray*}
so if we replace $a$ by $a+\theta^{n_0+k_0}\alpha$,
we can assume without loss of generality that
$\pi_{k+ec_1}$ is injective on $Ga$
and that $|Ga|=|G|$.
If we now let
$$
 A=A'\cup\{\sigma(a)+\theta^{k_0+c_1+j}:\sigma\in G,0\leq j<d/|G_x|\}
$$
where $G_x$ is the stabilizer of $x$, then $\pi_k|_{A}$ is $d$-to-$1$ onto $X=X'\cup Gx$ and $A$ is $G$-invariant with all orbits of length $|G|$. 
As 
$$
 k+ec_1\leq e(k_0+c_1+j)<k+c,
$$ 
we have that  $\pi_{k+c}|_{A}$ is injective.

Now, if $F$ is a $p$-adic field and if we chose $\alpha\in \mathcal{O}_E$ to be also a generator of $\mathcal{O}_E$ as a $\mathcal{O}_F$-algebra, by \cite[Chap.IV, Ex. 3(c)]{Serre}, one has the explicit bound $c_{0}\leq e(\mathfrak{P}/{p})/(p-1)$ which  implies the stated bound for $c$.
\end{proof}

\begin{remark}
Note that if (4) holds for some $c,k,A$, then
also for $c',k,A$ for any $c'\geq c$.
\end{remark}

\section{Proof of Theorem \ref{mainthm}}
\label{sec:main}
\noindent
Using the notation of Theorem \ref{mainthm}, for every $1\leq i\leq n$, let 
$\mathfrak{P}_i$ be the maximal ideal of $\mathcal{O}_{E_i}$,
$v_i$ the extension of $v_{\mathfrak{P}_i}$ to an algebraic closure of $E_i$,
$G_i={\rm Gal}(E_i/F_i)$ and 
$d=\prod_{i=1}^n|G_i|$.
Let $C$ be the constant from Theorem \ref{mainthm},
note that $C\geq 2$,
and let $c=4C^{n+1}$.
Fix an integer $\rho\geq 3C^n$
and note that $\rho/d\geq 3$ since $d\leq C^n$.
If $n=0$ let $\epsilon=0$, otherwise
fix $0<\epsilon<1$.

\subsection{Choosing the right degree}
\label{sec:degree}

If $n>1$ 
we apply Proposition \ref{Dir} to $x_i={q_i}^{f_i}$ 
to obtain positive integers $r>\rho/d$ and $k_1,\ldots,k_n$ such that for every $i$,
\begin{enumerate}[(i)]
\item \label{bound-r}
$r\leq q_i^{f_i k_i} \leq (1+\epsilon) r$ and
\item\label{b-ki} $k_i\leq 2^{2(n+1)}(\log C)^{n}\frac{\log(\rho/d)}{\log^n(1+\epsilon)}$,
\end{enumerate}
where we used that, for every $i$, $\log 2 \leq \log(x_i)=\log({q_i^{f_i}})\leq \log C$.
It follows that
\begin{equation*}\label{b-r} \log(\rho/d)\leq \log(r)\leq (4\log C)^{n+1}\frac{\log(\rho/d)}{\log^n(1+\epsilon)}. 
\end{equation*}
If $n=1$ we 
instead set $r=q_1^{f_1k_1}$, where $k_1=\lceil\log(\rho/d)/\log(q_1^{f_1})\rceil$, so that (\ref{bound-r}) holds with $\epsilon=0$, and 
\begin{equation*}\label{b-r-rho}
\log(\rho/d)\leq\log(r)\leq \log(\rho/d)+\log(q_1^{f_1}).
\end{equation*}
Using that $(4\log C)^{n+1}\geq\log^n(1+\epsilon)$, we conclude
\begin{equation}\label{eqn:deg}
 \rho \leq dr \leq \begin{cases} C\rho,&n=1\\
  \rho^{\frac{(4\log C)^{n+1}}{\log^n(1+\epsilon)}},&n>1 \end{cases}
\end{equation}

\subsection{Construction of the polynomial $g$}\label{constr-G}
We first want to prove the following: 
\begin{claim}\label{claim} 
For every $i$, there exists a polynomial $g_i\in \mathcal{O}_{K}[X]$ of degree $dr$ whose set of roots ${A_i}$ satisfies
\begin{enumerate}[(a)]
\item\label{aux1} $A_i\subseteq E_i$,
\item\label{aux2} $v_i({\alpha}-{\beta})< k_i+c$ for all ${\alpha},{\beta}\in{A_i}$ with ${\alpha}\neq{\beta}$, and
\item\label{aux3} $v({g_i}'(\alpha))\leq d\left(\frac{q_i^{f_i k_i}-1}{q_i^{f_i}-1}+ c\right)$ for every ${\alpha}\in{A_i}$.
\end{enumerate}
\end{claim}
\begin{proof}[Proof of the claim]
As 
$e(\mathfrak{P}_i|\mathfrak{p}_i)(d+|G_i|+\frac{e(\mathfrak{p}_i|p_i)}{p_i-1}+1)\leq C(C^n+C+C+1)\leq 4C^{n+1}=c$,
by Proposition \ref{A_i} there is
a $G_i$-invariant set $A_i'\subseteq\mathcal{O}_{E_i}$ with all orbits of length $|G_i|$ such that
${A_i'}\rightarrow\mathcal{O}_{E_i}/\mathfrak{P}_i^{k_i}$ is $d$-to-$1$
and ${A_i'}\rightarrow\mathcal{O}_{E_i}/\mathfrak{P}_i^{k_i+c}$
is injective.
As $|{A_i'}|=d q_i^{f_i k_i}$, $|G_i|$ divides $d$, and $r\leq q_i^{f_i k_i}$,
there exists a $G_i$-invariant subset ${\tilde A_i}\subseteq{A_i'}$
with $|\tilde{A_i}|=d r$.
Let 
$$
 \tilde{g_i}=\prod_{\alpha\in \tilde{A_i}}(X-\alpha)\in\mathcal{O}_{F_i}[X].
$$ 
We first prove that conditions \eqref{aux1}-\eqref{aux3} hold for $\tilde{g_i}$ and the set $\tilde{A_i}$, instead of $g_i$ and $A_i$.
Note that $\tilde{g_i}\in\mathcal{O}_{F_i}[X]$ is monic of degree $dr$ and that condition \eqref{aux1} holds for $\tilde{A_i}$ by construction. Moreover, as the map $\tilde{A_i}\rightarrow\mathcal{O}_{E_i}/\mathfrak{P}_i^{k_i+c}$ is injective, we have that condition \eqref{aux2} is also satisfied for $\tilde{A_i}$. 
As for condition \eqref{aux3},
note that the valuation $v_{\mathfrak{P}_i}$ on $\mathcal{O}_{E_i}$ induces 
a map $\bar{v}:(\mathcal{O}_{E_i}/\mathfrak{P}_i^{k_i})\setminus\{0\}\rightarrow\{0,\dots,k_i-1\}$ such that
${v}_{\mathfrak{P}_i}(\gamma)=\bar{v}(\pi_{k_i}(\gamma))$ for all $\gamma\in\mathcal{O}_{E_i}\setminus\mathfrak{P}_i^{k_i}$, where
$\pi_{k_i}$ denotes the residue map $\mathcal{O}_{E_i}\rightarrow\mathcal{O}_{E_i}/\mathfrak{P}_i^{k_i}$.
Now 
\begin{eqnarray*}
 v_{\mathfrak{P}_i}(\tilde{g_i}'(\alpha)) &=& \sum_{\alpha\neq\beta\in \tilde{A_i}}v_{\mathfrak{P}_i}(\alpha-\beta)\\&\leq& 
\sum_{\alpha\neq\beta\in A_i'}v_{\mathfrak{P}_i}(\alpha-\beta)\\
 &=&\sum_{\stackrel{\alpha\neq\beta\in A_i'}{\pi_{k_i}(\alpha)=\pi_{k_i}(\beta)}}v_{\mathfrak{P}_i}(\alpha-\beta)+\sum_{\stackrel{\alpha\neq\beta\in A_i'}{\pi_{k_i}(\alpha)\neq\pi_{k_i}(\beta)}}v_{\mathfrak{P}_i}(\alpha-\beta)\\
 &<&(d-1)\cdot(k_i+c)+d\cdot\sum_{ {0\neq a\in\mathcal{O}_{E_i}/\mathfrak{P}_i^{k_i}}}\bar{v}(a)
\end{eqnarray*}
and
\begin{eqnarray*}
 \sum_{{0\neq a\in\mathcal{O}_{E_i}/\mathfrak{P}_i^{k_i}}} \bar{v}(a) 
 &=& \sum_{j=0}^{k_i-1}|\{a:\bar{v}(a)=j\}|\cdot j
 \quad=\quad\sum_{j=0}^{k_i-1}\sum_{l=1}^j|\{a:\bar{v}(a)=j\}|\\
 &=&\sum_{l=1}^{k_i-1}\sum_{j=l}^{k_i-1}|\{a:\bar{v}(a)=j\}|
 \quad=\quad\sum_{l=1}^{k_i-1}|\{a:\bar{v}(a)\geq l\}|\\
 &=&\sum_{l=1}^{k_i-1}(q_i^{f_i(k_i-l)}-1)
 \quad=\quad \sum_{l=0}^{k_i-1}q_i^{f_i l} - k_i
 \quad=\quad \frac{1-q_i^{f_i k_i}}{1-q_i^{f_i}}-k_i
\end{eqnarray*}
and plugging this into the previous inequality gives condition (\ref{aux3}) for $\tilde{g}_i$.

As $\mathcal{O}_K$ is dense in $\mathcal{O}_{F_i}$ with respect to $v_i$,
we obtain a monic polynomial ${g_i}\in\mathcal{O}_K[X]$
of degree $dr$
arbitrarily close to $\tilde{g_i}$.
Let ${A_i}$ be the set of roots of ${g_i}$.
By the continuity of roots \cite[Theorem 2.4.7]{EP} we can achieve that the roots of ${g_i}$ are arbitrarily close to the roots of $\tilde{g_i}$, in particular that conditions \eqref{aux2} and \eqref{aux3} are satisfied by $g_i$ and $A_i$.
Moreover, by Krasner's lemma \cite[Ch.II, \S 2, Proposition 4]{Lang}, we can in addition achieve condition \eqref{aux1}, completing the proof of the claim.
\end{proof}

Now, let $p_0$ be the smallest prime number not in the set $\{p_1,\ldots,p_n\}$ and let $\mathfrak{p}_0$ be a prime ideal of $\mathcal{O}_K$ above $p_0$. Fix
a monic polynomial $g_0\in\mathcal{O}_K[X]$ of degree $dr$ 
whose reduction modulo $\mathfrak{p}_0$ is irreducible.
Let 
\begin{equation}\label{defmi}
 m_i=\frac{d}{e_i}\left(\frac{q_i^{f_i k_i}-1}{q_i^{f_i}-1}+k_i+2c\right)
\end{equation}
and 
$$
 \mathfrak{a}=\mathfrak{p}_0\mathfrak{p}_1^{m_1}\cdots\mathfrak{p}_n^{m_n}.
$$ 
By the Chinese Remainder Theorem and Proposition \ref{lem:small_rep} 
there exists a monic polynomial $g\in\mathcal{O}_K[X]$ such that
\begin{enumerate}
\item\label{degg_dr} $\deg g=dr$,
\item\label{cond_p0} $g\equiv g_0\mbox{ mod }\mathfrak{p}_0[X]$, 
\item\label{cond_equiv} $g\equiv {g_i}\mbox{ mod }\mathfrak{p}_i^{m_i}[X]$ for $i=1,\dots,n$, and
\item\label{cond_coeff} $|\sigma(a)|\leq\delta_K N(\mathfrak{a})^{1/[K:\mathbb{Q}]}$
for every coefficient $a$ of $g$ and every $\sigma\in{\rm Hom}(K,\mathbb{C})$,
\end{enumerate}
where $\delta_K=[K:\Q]^{\frac{3}{2}}{2}^{\frac{[K:\Q]([K:\Q]-1)}{2}}\sqrt{|\Delta_K|}$.

Note that (\ref{cond_p0}) implies that $g$ is irreducible.
In particular, we get from (\ref{degg_dr}) and (\ref{eqn:deg})
that every root $\alpha$ of $g$ satisfies
the degree bound
(\ref{eqn:thm2}) of Theorem \ref{mainthm}.

\subsection{The roots of $g$ are in $E_i$ for every $i$.}\label{roots-g}
We claim that the conditions of Proposition \ref{val-BZ} hold for the field $E_i$, the polynomial $g$ and 
$x_0=\alpha\in {A_i}$ (which lies in $E_i$ by condition  \eqref{aux1})  by setting 
 $a=v_i({g_i}'(\alpha))$ and $b=k_i+c-1$.
Indeed, by \eqref{aux3}
\begin{eqnarray}\label{aem}
 a=v_i({g_i}'(\alpha))\leq  d\cdot\frac{q_i^{f_i k_i}-1}{q_i^{f_i}-1}+dc<e_im_i-b,
\end{eqnarray}
and, writing $g-{g_i}= t_i$ with $t_i\in\mathfrak{p}_i^{m_i}[X]$ by (\ref{cond_equiv}) of Section \ref{constr-G}, we have  $g(\alpha)=t_i(\alpha)$ and therefore $v_i(g(\alpha))\geq e_im_i>a+b$,
so condition \eqref{ci} holds.
Similarly for condition \eqref{cii}, we have $g'(\alpha)=t_i'(\alpha)+{g_i}'(\alpha)$
and since \[v_i(t_i'(\alpha))\geq e_im_i> v_i({g_i}'(\alpha)),\]  we conclude that
$v_i(g'(\alpha))=v_i({g_i}'(\alpha))=a$.
Now for $\nu\geq 2$ write
$$
 {g_i}^{(\nu)}(\alpha) = \nu! {g_i}'(\alpha)\sum_{\substack{B\subseteq {A_i}\\|B|=\nu-1\\ \alpha\notin B}}\prod_{\beta\in B}(\alpha-\beta)^{-1}.
$$
Thus
$$ 
 v_i({{g_i}^{(\nu)}(\alpha)}/{\nu!})\geq a-(\nu-1)\max_{\beta\neq\alpha}v_i(\alpha-\beta) {\geq} a-(\nu-1)b
$$
where the last inequality holds by  {(\ref{aux2})}.
Moreover, using (\ref{aem}), we get
\[v_i({t_i^{(\nu)}(\alpha)}/{\nu!})\geq e_im_i-v_i(\nu!)\geq a+b-\frac{e(\mathfrak{P}_i|p_i)\nu}{p_i-1}\geq a-(\nu-1)b\]
where the last inequality holds since 
$b\geq c\geq \frac{e(\mathfrak{P}_i|p_i)}{p_i-1}$.
Thus 
\[v_i({g^{(\nu)}(\alpha)}/{\nu!})\geq \min\left\{v_i({{g_i}^{(\nu)}(\alpha)}/{\nu!}), v_i({t_i^{(\nu)}(\alpha)}/{\nu!})\right\}\geq a-(\nu-1)b\] 
fulfilling condition \eqref{ciii}.

So Proposition \ref{val-BZ} gives ${\alpha}'\in E_i$ with $g({\alpha}')=0$ and $v_i(\alpha'-\alpha)>b$.
As $v_i(\alpha-\beta)\leq b$ for all $\beta\in {A_i}\setminus\{\alpha\}$
by (\ref{aux2}),
we conclude that ${\alpha}'\neq{\beta}'$ for all $\alpha\neq\beta$.
Hence $g$ has precisely $|{A_i}|=dr$ many roots in $E_i$.
As this holds for every $i$ and
$g$ totally splits in the maximal Galois extension $L$ of $K$ that is contained in all $E_i$. Moreover, as $g\in \mathcal{O}_K[X]$, all roots of $g$ are actually in $\mathcal{O}_L$.
\subsection{Bounding the height of the roots of $g$}\label{b-root}
From condition (\ref{cond_coeff}) of Section \ref{constr-G}, for every coefficient $a$ of $g$ and every $\sigma\in{\rm Hom}(K,\mathbb{C})$, we have
$$
 |\sigma(a)|\leq B:=\delta_K N(\mathfrak{p}_0)^{1/[K:\mathbb{Q}]}\cdot\prod_{i=1}^n N(\mathfrak{p}_i)^{m_i/[K:\mathbb{Q}]}.
$$ 
By Proposition \ref{H-min-root}, $g$ has a root $\alpha\in\mathcal{O}_L$ with $h(\alpha)$ bounded by
\begin{equation}\label{bo-h}
\frac{\log(B\sqrt{\deg g+1})}{\deg g} \leq 
\frac{\log\left(\delta_K N(\mathfrak{p}_0)^{1/[K:\mathbb{Q}]}\sqrt{\deg g+1}\right)}{\deg g}+\sum_{i=1}^n \frac{m_i}{\deg g}\cdot\frac{\log (q_i)}{[K:\mathbb{Q}]}.
\end{equation}
By the definition of $m_i$ in (\ref{defmi}) and recalling $\deg g=dr$ from \eqref{degg_dr}, we have
\begin{eqnarray}\label{eq-mi}
 \frac{m_i}{\deg g} &=&\frac{1}{e_i}\cdot\frac{1}{q_i^{f_i}-1}\cdot\frac{q_i^{f_i k_i}-1}{r}+\frac{k_i}{e_i r}+\frac{2c}{e_i r}.
\end{eqnarray}
Condition \eqref{bound-r} of Section \ref{constr-G} implies that
\[
 \frac{q_i^{f_i k_i}-1}{r}\leq 1+\epsilon.
\] 
Moreover,
\[
 \frac{k_i}{e_i r}=\frac{\log(q_i^{f_i k_i})}{e_i r f_i \log(q_i)}\leq \frac{2d}{e_i f_i \log(q_i)}\cdot \frac{\log(\deg g)}{\deg g}\leq 3 C^n \frac{\log(\deg g)}{\deg g}
 \] 
Finally,
$$
 \frac{2c}{e_i r}\leq \frac{8d C^{n+1}}{e_i \deg g}\leq \frac{8C^{2n+1}}{\deg g}.
$$
Therefore, substituting in \eqref{eq-mi}, and recalling that $C\geq 2$ and $\rho\geq 3$, we have
\[ 
 \frac{m_i}{\deg g} \leq \frac{1}{e_i(q_i^{f_i}-1)}+\frac{\epsilon}{e_i(q_i^{f_i}-1)}+11 C^{2n+1}\frac{\log(\deg g)}{\deg g}.
\]
Thus the second summand in \eqref{bo-h} can be bounded as
\begin{align*}
 \sum_{i=1}^n \frac{m_i}{\deg g}\cdot\frac{\log(q_i)}{[K:\mathbb{Q}]}\leq \sum_{i=1}^n& \frac{f(\mathfrak{p}_i|p_i)}{[K:\mathbb{Q}]}\cdot\frac{\log(p_i)}{e_i(q_i^{f_i}-1)}+n\epsilon+11nC^{2n+2}\frac{\log(\deg g)}{\deg g}.
\end{align*}
As for the first summand in \eqref{bo-h}, note that $N(\mathfrak{p}_0)^{1/[K:\Q]}\leq p_0$ where $p_0$ is the smallest prime number not in the set $\{p_1,\ldots,p_n\}$, which, by Bertrand's postulate, can be bounded by $p_0<2\max_i p_i\leq 2C$.
Moreover, 
\[
 \delta_K=[K:\Q]^{\frac{3}{2}}{2}^{\frac{[K:\Q]([K:\Q]-1)}{2}}\sqrt{|\Delta_K|}\leq C^2 \cdot 2^{\frac{C(C-1)}{2}}
\]
and thus, as $C\geq 2$,
\begin{eqnarray*}
\frac{\log\left(\delta_K N(\mathfrak{p}_0)^{1/[K:\mathbb{Q}]}\sqrt{\deg g+1}\right)}{\deg g}
&\leq& \frac{\log\left(C^3 2^{\frac{C^2-C+2}{2}}\sqrt{\deg g+1}\right)}{\deg g}\leq 2nC^{2n+1}\frac{\log(\deg g)}{\deg g}.
\end{eqnarray*}
Therefore, from \eqref{bo-h} we get
\[
 h(\alpha)\leq \sum_{i=1}^n \frac{f(\mathfrak{p}_i|p_i)}{[K:\mathbb{Q}]}\cdot\frac{\log(p_i)}{e_i(q_i^{f_i}-1)}+n\epsilon+13 nC^{2n+2}\frac{\log(\deg g)}{\deg g},
\]
so $\alpha$ satisfies the height bound (\ref{eqn:thm1}) of 
Theorem \ref{mainthm}
(recalling that $\epsilon=0$ if $n=1$).

\section*{Acknowledgments} 
\noindent The authors thank 
Lukas Pottmeyer for pointing out the results in \cite{Fil,FiliPetsche,FP},
Paul Fili for the exchange regarding \cite{Fil},
and Philip Dittmann for helpful discussions on $p$-adic fields
as well as for suggesting the short proof of Proposition \ref{val-BZ}.
The first author's work has been funded by the ANR project Gardio 14-CE25-0015.


\begin{thebibliography}{BFH17}

\bibitem[BFH17]{BFH}
J.~Biasse, C.~Fieker and T.~Hofmann.
\newblock On the computation of the HNF of a module over the ring of integers of a number field.
\newblock{\em J.\ Symb.\ Comp.} 80(3):581--615, 2017.

\bibitem[BZ01]{BZ} 
E.~Bombieri and U.~Zannier. 
\newblock A note on heights in certain infinite extensions of $\Q$.
\newblock {\em Rend.\ Mat.\ Acc.\ Lincei} 12:5--15, 2001.

\bibitem[CSV12]{Cha} 
X.-W. Chang, D.~Stehl\'e and G.~Villard.  
\newblock Perturbation analysis of the QR factor R in the context of LLL lattice basis reduction.
\newblock {\em Math.\ Comput.} 81(279):1487--1511, 2012.

\bibitem[CF21]{CF} 
S.~Checcoli and A.~Fehm.
\newblock On the Northcott property and local degrees. 
\newblock To appear in {\em Proc.\ Amer.\ Math.\ Soc.}, 2021.

\bibitem[EP05]{EP} 
A.\ J.\ Engler and A.\ Prestel. 
\newblock {\em Valued Fields}. 
\newblock Springer, 2005.

\bibitem[Fil14]{Fil} 
P.~Fili.
\newblock {On the heights of totally $p$-adic numbers}. 
\newblock {\em J.\ Th\'eor.\ Nombres Bordeaux} 26(1):103--109, 2014.

\bibitem[FP15]{FiliPetsche}
P.~Fili and C.~Petsche.
\newblock Energy Integrals Over Local Fields and Global Height Bounds.
\newblock {\em International Mathematics Research Notices} 2015(5):1278--1294, 2015.

\bibitem[FP19]{FP}
P.~A.~Fili and L.~Pottmeyer.
\newblock Quantitative height bounds under splitting conditions.
\newblock {\em Trans.\ Amer.\ Math.\ Soc.} 372:4605--4626, 2019.

\bibitem[Fla96]{Fla96}
V.~Flammang. 
\newblock Two new points in the spectrum of the absolute Mahler measure of totally positive algebraic integers.
\newblock {\em Math.\ Comp.} 65(213):307--311, 1996.

\bibitem[Lan94]{Lang} 
S.~Lang. 
\newblock {\em Algebraic Number Theory}.
\newblock Second Edition. Springer, 1994.

\bibitem[NSV11]{NSV} 
A.~Novocin, D.~Stehl\'e and G.~Villard.
\newblock An LLL-reduction algorithm with quasi-linear time complexity.
\newblock In: STOC’11--Proceedings of the 43rd
ACM Symposium on Theory of Computing. ACM, New York, pp. 403--412, 2011.

\bibitem[PS19]{PS19} 
C.~Petsche and E.~Stacy.
\newblock A dynamical construction of small totally p-adic algebraic numbers.
\newblock {\em J.\ Number Theory} 202:27--36, 2019.

\bibitem[Pot15]{Pot} 
L.~Pottmeyer.
\newblock Heights and totally $p$-adic numbers.
\newblock {\em Acta Arith.} 171(3):277--291, 2015.

\bibitem[Sch73]{Sch} 
A.~Schinzel. 
\newblock On the product of the conjugates outside the unit circle of an algebraic number.
\newblock {\em Acta Arith.} 24:385--399, 1973. Addendum, ibidem, 26, 329--361, 1973.

\bibitem[Sch80]{Schm} 
W.~M.~Schmidt.
\newblock {\em Diophantine Approximation}.
\newblock Springer, 1980.

\bibitem[Ser80]{Serre} 
J.-P.~Serre.
\newblock {\em Local Fields}.
\newblock Springer, 1980.

\bibitem[Smy80]{Smy80} 
C.~J.~Smyth.
\newblock On the measure of totally real algebraic integers. 
\newblock {\em J.\ Austral.\ Math.\ Soc.\ Ser.\ A} 30(2):137--149, 1980.

\bibitem[Smy81]{Smy81} 
C.~J.~Smyth.
\newblock On the measure of totally real algebraic integers.\ II.
\newblock {\em Math.\ Comp.} 37(155):205--208, 1981.

\bibitem[Smy07]{Smy07} 
C.~J.~Smyth.
\newblock The Mahler measure of algebraic numbers: a survey.
\newblock In:  McKee, Smyth (eds.), {\em Number Theory and Polynomials}.  Cambridge University Press. pp.~322--349, 2007. 

\bibitem[Zan09]{Zan} 
U.~Zannier.
\newblock {\em Lecture Notes on Diophantine Analysis (with an Appendix by F. Amoroso)}.
\newblock Edizioni della Normale, 2009. 

\end{thebibliography}
\end{document}